\documentclass[a4paper,draft,reqno,12pt]{amsart}
\usepackage[english]{babel}
\usepackage{amsmath}
\usepackage{amssymb}
\usepackage{amscd}
\usepackage{amsthm}
\usepackage{euscript}
\usepackage{tikz}

\newtheorem{lemma}{Lemma}
\newtheorem{theorem}{Theorem}
\newtheorem{question}{Question}
\newtheorem{corollary}{Corollary}
\theoremstyle{definition}
\newtheorem{definition}{Definition}

\theoremstyle{remark}

\DeclareMathOperator{\Aut}{Aut}

\def\Ker{{\rm Ker}\,}

\def\ZZ{{\mathbb Z}}
\def\BG{{\mathbb G}}
\def\KK{{\mathbb K}}

\def\BG{{\mathbb G}}

\def\BK{{\mathbb K}}

\def\BZ{{\mathbb Z}}

\def\QQ{{\mathbb Q}}

\def\ML{\mathrm{ML}}
\def\HD{\mathrm{HD}}

\def\SAut{\mathrm{SAut}}

\title{Modified Makar-Limanov and Derksen invariants}

\thanks{The paper was supported by RSF grant 22-41-02019.}
\author{Sergey Gaifullin}
\address{Lomonosov Moscow State University, Faculty of Mechanics and Mathematics, Department of Higher Algebra, Leninskie Gory 1, Moscow, 119991 Russia;\linebreak 
Moscow Center for Fundamental and Applied Mathematics, Moscow, Russia; \linebreak 
and \linebreak
National Research University Higher School of Economics, Faculty of Computer Science, Pokrovsky Boulevard 11, Moscow, 109028, Russia}
\email{sgayf@yandex.ru}

\author{Anton Shafarevich}
\address{Moscow Center for Fundamental and Applied Mathematics, Moscow, Russia; \linebreak 
and \linebreak
National Research University Higher School of Economics, Faculty of Computer Science, Pokrovsky Boulevard 11, Moscow, 109028, Russia}
\email{shafarevich.a@gmail.com}

\begin{document}
\maketitle

\begin{abstract}
    We investigate modified Makar-Limanov and Derksen invariants of an affine algebraic variety. The modified Makar-Limanov invariant is the intersection of kernels of all locally nilpotent derivations with slices and the modified Derksen invariant is the subalgebra generated by these kernels. We prove that modified Makar-Limanov invariant coincide with Makar-Limanov invariant if there exists a locally nilpotent derivation with a slice. Also we construct an example of a variety admitting a locally nilpotent derivation with a slice such that modified Derksen invariant does not coincide with Derksen invariant.
\end{abstract}

\section{Introduction}

Let $\BK$ be an algebraically closed field of characteristic zero. Let $X$ be an affine irreducible algebraic variety over $\BK.$ We denote by $\BK[X]$ the algebra of regular functions on $X$. A linear mapping $\partial\colon \BK[X]\rightarrow \BK[X]$ is called a {\it derivation} if it satisfies the Leibniz rule: $\partial(fg)=f\partial(g)+g\partial(f)$. A derivation $\partial: \BK[X]\rightarrow \BK[X]$ is called \emph{locally nilpotent} (LND) if for each $f \in \BK[X]$ there is $m \in \BZ_{>0}$ such that $\partial^m(f) = 0.$ We denote the set of all locally nilpotent derivations on $X$ by $\mathrm{LND}(X).$ Let $\Aut(X)$ be the group of regular automorphisms of $X.$ Then for each locally nilpotent derivation $\partial$ and $t \in \BK$ there is an automorphism $\mathrm{exp}(t\partial)$ of $\BK[X]$ (and respectively $X$) which is given by the formula

$$\mathrm{exp}(t\partial)(f) = \sum_{i = 0}^{\infty} \frac{t^i\partial^i(f)}{i!}.$$
The map $t \to \mathrm{exp}(t\partial)$ defines a homomorphism $\BG_a \to \Aut(X)$ where $\BG_a = (\BK, +)$ is the additive group of the filed $\BK.$  Denote the image of this homomorphism by $H_\partial$. Such subgroups of $\Aut(X)$ we call $\BG_a$-subgroups.  It is known that every homomorphism $\BG_a \to \Aut(X)$ is obtained this way; see \cite{FR} for details. If $\partial \in \mathrm{LND}(X)$ then the kernel of $\partial$ 
$$\mathrm{Ker}\ \partial = \{f\ |\ \partial(f) = 0\}$$
coincides with the set of regular invariants $\BK[X]^{H_\partial}$ with respect to the natural action of the group $H_\partial\subseteq \mathrm{Aut}(X)$ on $\BK[X].$  The subgroup $\SAut(X)$ in $\Aut(X)$ generated by all $\BG_a$-subgroups in $\Aut(X)$ is called subgroup of {\it special automorphisms} of $X$.

In \cite{ML} Makar-Limanov introduced an invariant $\ML(X)$ of a variety $X$. This invariant is a  subalgebra in $\BK[X]$ equal to the intersection of  kernels of all LNDs on $X$.
$$
\ML(X) = \bigcap_{\partial \in \mathrm{LND}(X)} \mathrm{Ker}\,\partial.
$$
It is easy to see that the Makar-Limanov invariant coincides with the algebra of $\SAut(X)$-invariants $\BK[X]^{\SAut(X)}$. 
The Makar-Limanov invariant allowed to distinguish Koras-Russell cubic $\{x+x^2y+z^2+t^3=0\}$ and $\BK^3$. It became a powerful tool to distinguish nonisomorphic varieties. Also it can be used for investigating automorphism group of a variety, see for example, \cite{ML2,MJ,P,G}. 

In \cite{HD} Derksen introduced an alternative invariant based on LNDs. Derksen invariant is the subalgebra in $\BK[X]$ generated by kernels of all nonzero LNDs.
$$
\HD(X)=\BK\left[\mathrm{Ker}\partial\mid\partial\in \mathrm{LND}(X)\!\setminus\!\{0\}\right].
$$
This invariant also distinguishes the Koras-Russell cubic and the affine space. In \cite{CM} Crachiola and Maubach investigate the question if one of these invariants distinguish varieties better then another. They give examples with trivial one invariant and nontrivial another. 

For some LNDs there exist elements, which are called slices.
\begin{definition}
    Let $\partial$ be a locally nilpotent derivation. We say that a function $s\in \BK[X]$ is a \emph{slice} with respect to an LND $\partial$ if $\partial(s) = 1.$ We denote by $\mathrm{LND}^*(X)$ the set of all locally nilpotent derivations of $\BK[X]$ which have a slice. 
\end{definition}
The well-known Slice theorem says that if $s$ is a slice of $\partial$, then $\BK[X]=(\mathrm{Ker}\,\partial)[s]$.  

In Section~11.9 of the first edition of the book~\cite{FR}, Froudenburg suggested to consider the following modifications of Makar-Limanov and Derksen invariant. In~\cite{DG}, $\ML^*(X)$ is called Makar-Limanov-Freudenburg invariant.

\begin{definition}
    Let $X$ be an affine variety. Suppose that $\mathrm{LND}^*(X) \neq \emptyset$. Then the \emph{modified Makar-Limanov} invariant is the algebra:
    $$\ML^*(X) = \bigcap_{\partial \in \mathrm{LND}^*(X)} \mathrm{Ker}\ \partial \subseteq \BK[X].$$

    The \emph{modified Derksen invariant} of an affine variety $X$ is the subalgebra in $\mathbb{K}[X]$ generated by all kernels of LNDs having slices.
$$\HD^*(X)=\BK[\mathrm{Ker}\,D\mid D\in \mathrm{LND}^*(X)\!\setminus\!\{0\}].$$
\end{definition}
It is easily follows from the definition that $\ML(X)\subseteq \ML^*(X)$ and $\HD^*(X)\subseteq \HD(X)$.

In \cite{DG} characterizations of affine two space and affine three space in terms of $\ML^*(X)$ were obtained. Let $X$ be a $2$-dimensional variety or a $3$-dimensional factorial variety. Then the following conditions are equivalent, see~\cite[Theorems~3.8 and 4.6]{DG}:
\begin{enumerate}
\item $X$ is the affine space,

\item $\ML^*(X)=\BK$,

\item $\ML(X)=\BK$ and $\ML^*(X)\neq \BK[X]$.
\end{enumerate}

In this paper we prove that when $\mathrm{LND}^*(X)\neq\emptyset$, then $\ML^*(X)$ is always equal to $\ML(X)$. Examples, when $\mathrm{LND}^*(X)=\emptyset$, i.e. $\ML^*(X)=\BK[X]$, and $\ML(X)\neq \BK[X]$ are given in \cite{DG}. In partiqular, our result answers the Question~5.9 from \cite{DG}. Let $X$ be the Koras-Russell cubic. In~\cite{D} Dubouloz showed that $\ML(X\times\BK)=\BK$. So, we have $\ML^*(X)=\BK$. This mean that $X\times\BK$ and $\BK^4$ can not be distinguished by $\ML^*$.

Our result means that considering the modified Makar-Limanov invariant we do not obtain new invariant subalgebra of $\BK[X]$. The other situation we have with modified Derksen invariant. We construct an example of such a variety that $\ML(X)=\ML^*(X)=\BK$, $\mathrm{HD}(X)=\BK[X]$, but $\mathrm{HD}^*(X)$ is a proper subalgebra of $\BK[X]$. In our example the variety $X$ is a nonnormal toric variety.

The reason of this situation is that $\ML^*(X)$ can be interpreted as a subalgebra of invariants for some subgroup $\SAut^*(X)\subseteq\Aut(X)$. But using results of the paper~\cite{5A} it can be shown, that generic orbits of this subgroup coincides with generic $\SAut(X)$-orbits and this implies equality of $\ML^*(X)$ and $\ML(X)$. But the Derksen and modified Derksen invariants are not invariants of some group action.

The authors are grateful to Nikhilesh Dasgupta for fruitfull discussions. The first author is a Young Russian Mathematics award winner and would like to thank its sponsors and jury.

\section{Preliminaries}\label{toric}

\subsection{Toric varieties}

In this section we give basic facts about toric varieties. More information about toric varieties one can find in books~\cite{CLSch} and~\cite{Ful}. An irreducible algebraic variety is called {\it toric}, if an algebraic torus $T=(\mathbb{K}^\times)^n$ algebraically acts on it with an open orbit. We can assume the action of $T$ on $X$ to be effective. Note that we do not assume toric variety to be normal. Let $X$ be affine. An affine toric variety $X$ corresponds to a finitely generated monoid $P$ of weights of $T$-semiinvariant regular functions. Let us identify the group of characters $\mathfrak{X}(T)$ with a free abelian group $M=\mathbb{Z}^n$. A vector $m\in\mathbb{Z}^n$ with integer coordinates corresponds to the character~$\chi^m$. Since the open orbit on $X$ is isomorphic to $T$, we have an embedding of algebras of regular functions $\mathbb{K}[X]\hookrightarrow\mathbb{K}[T]$. Identifying the algebra $\mathbb{K}[X]$ with its image we obtain the following subalgebra graded by $P$
$$
\mathbb{K}[X]=\bigoplus_{m\in P}\mathbb{K}\chi^m\subset \bigoplus_{m\in M}\mathbb{K}\chi^m=\mathbb{K}[T].
$$
Let us consider the vector space $M_{\mathbb{Q}}=M\otimes_{\mathbb{Z}}\mathbb{Q}$ over the field of rational numbers. The monoid $P$ generates the cone $\sigma^{\vee}=\mathbb{Q}_{\geq0}P\subset M_{\mathbb{Q}}$. Since $P$ is finitely generated, the cone $\sigma^\vee$ is a finitely generated polyhedral cone. Since the action of~$T$ on $X$ is effective, the cone $\sigma^\vee$ does not belong to any proper subspace of $M_\QQ$. The variety $X$ is normal if and only if the monoid $P$ is {\it saturated}, i.e. $P=M\cap \sigma^{\vee}$. If $P$ is saturated, then the monoid $P_{sat}=M\cap \sigma^{\vee}$ we call the {\it saturation}  of the monoid $P$. Elements of $P_{sat}\setminus P$ we call {\it holes} of $P$. Let us give some definition according to \cite{TY}.

\begin{definition}
An element $p$ of the monoid $P$ is called {\it saturation point} of $P$, if the moved cone $p+\sigma^{\vee}$ has no holes, i.e. $(p+\sigma^{\vee})\cap M\subset P$. 
 
A face $\tau$ of the cone $\sigma^{\vee}$ is called {\it almost  saturated}, if there is a saturation point of  $P$ in $\tau$. 
Otherwise $\tau$ is called a {\it nowhere saturated} face.
\end{definition}

The lattice of one-parameter subgroups of the torus $T$ we denote by $N$. The lattice $N$ is a dual lattice to $M$. There is a natural pairing
$M\times N\rightarrow \mathbb{Z}$, which we denote $\langle \cdot,\cdot\rangle$. This pairing can be extended to a pairing between vector spaces $N_\QQ=N\otimes_\ZZ\QQ$ and $M_\QQ$. In the space $N_\QQ$ we define the cone $\sigma$ dual to $\sigma^\vee$, by the rule
$$
\sigma=\{v\in N_\QQ\mid\forall w\in\sigma^\vee : \langle w,v\rangle\geq 0\}.
$$
The finitely generated polyhedral cone $\sigma$ is {\it pointed}, i.e. it does not contain any nontrivial subspaces.

There is a bijection between $k$-dimensional faces of $\sigma$ and $(n-k)$-dimensional faces of~$\sigma^\vee$. A face $\tau\preccurlyeq\sigma$ corresponds to the face  $\widehat{\tau}=\tau^{\bot}\cap\sigma^\vee\preccurlyeq \sigma^\vee$. Also there is a bijection between $(n-k)$-dimensional faces of $\sigma^\vee$ and $k$-dimensional $T$-orbits on $X$. A face $\widehat{\tau}\preccurlyeq \sigma^\vee$ corresponds to the orbit, which is open in the set of zeros of the ideal 
$$I_{\widehat{\tau}}=\bigoplus_{m\in P\setminus\widehat{\tau}}\KK\chi^m.$$
The composition of these bijections gives a bijection between $k$-demensional faces of the cone $\sigma$ and $k$-dimensional $T$-orbits. The orbit corresponding to a face $\tau$, we denote by~$O_\tau$.

\subsection{Locally nilpotent derivations on toric varieties and Demazure roots}
In this section we skip some proofs. The proofs can be found in books \cite{FR} for facts about LNDs, \cite{Ful} and \cite{CLSch} for facts about toric varieties.

Let $X$ be an affine irreducible variety and $A$ be an abelian group. Consider $A$-grading on $\BK[X]$

$$\BK[X] = \bigoplus_{a \in A}\BK[X]_a, \ \ \ \BK[X]_a\BK[X]_b \subseteq \BK[X]_{a+b}.$$

A derivation $\partial : \BK[X] \rightarrow \BK[X]$ is called {\it $A$-homogeneous of degree $a_0 \in A$} if for all $f \in \BK[X]_a$ we have $\partial(f) \in \BK[X]_{a+a_0}.$

If we have a $\ZZ$-grading, then each LND $\delta$ can be decomposed onto homogeneous derivations (components) $\delta=\delta_l+\ldots+\delta_k$. And $\delta_l$ and $\delta_k$ are LNDs. This implies that if we have $\ZZ^n$-grading, then each LND can be decomposed onto a finite sum of homogeneous derivations. And the derivations with degrees corresponding to vertices of the convex hull of all degrees are LNDs.

Let $X$ be a toric variety. In~\cite{De} all $T$-normalized $\BG_a$-subgroups are discribed. In affine case these subgroups correspond to $M$-homogeneous LNDs. They were discribed in~\cite{L} and correspond to so called Demazure roots. Let us remind this description. 
Consider a finitely generated cone $\sigma$.
Let $$\mathfrak{R}_{\rho}:=\{e \in M \ \ | \ \ \langle e,v_{\rho} \rangle = -1, \langle e,v_{\rho '} \rangle \geq 0 \ \ \forall \rho ' \neq \rho \in \sigma(1)\}.$$
Then the elements of the set $\mathfrak{R} := \bigsqcup\limits_{\rho}\mathfrak{R}_{\rho}$ are called the {\it Demazure roots} of the cone~$\sigma$.
We will call ray $\rho$  the {\it distinguished ray} of the Demazure root $e$ if $e \in \mathfrak{R}_{\rho}$.

For a normal affine toric variety $X$ each Demazure root corresponds to an $M$-homogeneous LND of $\KK[X]$ given by 
\begin{equation}\label{demaz}\partial_e(\chi^m) = \langle p_{\rho}, m \rangle \chi^{e+m}.\end{equation}
Each $M$-homogeneous LND on $X$ has the form $\lambda\partial_e$ for some Demazure root $e$.

In~\cite{BG} all $M$-homogeneous LND on $\BK[X]$ are described. All $M$-homogeneous LND on $\BK[X]$ are $\partial_e$ for some Demazure root $e$, but not all $\partial_e$ give a well defined LND on $\BK[X]$. It occurs, that $\partial_e$ is well defined on $\BK[X]$ if and only if $(P+e)\cap P_{sat}\subseteq P$. The fact that there are no other LNDs on $X$ can be easily obtained for example from \cite[Lemma 6]{BGSh}. This implies the following lemma, see~\cite[Lemma~4]{BG}.

\begin{lemma}\label{eu}
Let $X$ be a toric variety, $\sigma$ be the corresponding cone and $\rho$ be an extremal ray of~$\sigma$. Denote by $O_{\rho}$ the corresponding orbit. Then the following conditions are equivalent:

1) the face $\widehat{\rho}$ of the cone $\sigma^{\vee}$ is almost saturated;

2) there exist a Demazure root $e\in \mathfrak{R}_\rho$ of the cone $\sigma$ such, that the corresponding derivation $\delta_e$ of the algebra $\mathbb{K}[X]$ is well defined;

3) the orbit $O_{\rho}$ consists of smooth points.
\end{lemma}

It is easy to see that 
$$\Ker \partial_e=\bigoplus_{m\in P\cap {\rho_i}^\bot} \BK\chi^m.$$

\section{Modified Makar-Limanov invariant}\label{MML}

In this section we prove that if a variety $X$ admits an LND with slice, then $\ML^*(X)=\ML(X)$. To prove this we use some notations and assertions from~\cite{5A}.

\begin{definition}
    We say that a subgroup $G$ in $\Aut(X)$ is \emph{algebraically generated} if it is generated as an abstract group by a family of connected algebraic subgroups of $\Aut(X).$
\end{definition}

Let us remind that we denote by $\mathrm{SAut}(X)$ the subgroup in $\Aut(X)$ generated by all $\BG_a$-subgroups in $\Aut(X)$. If $\mathrm{LND}^*(X) \neq \emptyset$ we denote by $\SAut^*(X)$ the subgroup in $\Aut(X)$ generated by all subgroups of the form $H_\partial$ where $\partial \in \mathrm{LND}^*(X).$ It is easy to see that $\ML(X)=\BK[X]^{\SAut(X)}$ and $\ML^*(X)=\BK[X]^{\SAut^*(X)}$.
We see that $\SAut(X)$ and $\SAut^*(X)$ are algebraically generated groups. By \cite[Proposition 1.3]{5A} orbits in $X$ of algebraically generated group are locally closed.

\begin{definition}
    Let $G \subseteq \SAut(X)$ be a $\BG_a$-generated subgroup and let $\Omega \subseteq X$ be a subset invariant under the $G$-action. We say that locally nilpotent  derivation $\partial$ with associated one-paramter subgroup $H = \mathrm{exp}(\BK \partial)$ satisfies the \emph{orbit separation property} on $\Omega$, if there is an $H$-stable subset $U(H) \subseteq \Omega$ such that
    \begin{enumerate}
        \item for each $G$-orbit $O$ contained in $\Omega, $ the intersection $U(H) \cap O$ is open and dense;
        \item the global $H$-invariants $\BK[X]^H = \mathrm{ker}\ \partial$ separate one-dimensional $H$-orbits in $U(H).$ 
    \end{enumerate}
\end{definition}

\begin{definition}
    Let $X$ be an irreducible affine variety. A set $\mathcal{N}$ of locally nilpotent derivations on $X$ is said to be \emph{saturated} if it satisfies the following two conditions. 
    \begin{enumerate}
        \item $\mathcal{N}$ is closed under conjugation by elements in $G$, where $G$ is the subgroup of $\SAut(X)$ generated by $\mathcal{N}.$
        
        \item $\mathcal{N}$ is closed under taking replicas, i.e. for all $\partial \in \mathcal{N}$ and $f \in \mathrm{ker}\ \partial$ we have $f\partial \in \mathcal{N}.$
    \end{enumerate}
\end{definition}

Let $G \subseteq \SAut(X)$ be $\BG_a$-generated subgroup generated by a saturated set $\mathcal{N}$ of locally nilpotent derivations. Let $\Omega \subseteq X$ be a $G$-stable subset. Suppose that there are locally nilpotent derivations $\partial_1, \ldots, \partial_s \in \mathcal{N}$ satisfy the following two conditions.
\begin{enumerate}
\item The tangent vectors to the orbits of the groups $\mathrm{exp}(\BK\partial_i)$ with $i = 1,\ldots, s$ span $T_x(Gx)$ for every point $x \in \Omega;$
\item $\partial_i$ has the orbit separation property on $\Omega$ for all $i = 1,\ldots,s.$
\end{enumerate}

\begin{lemma}\cite[Corollary 2.12]{5A}
With the notations and assumptions  as above, for every $G$-orbit $O\subseteq \Omega$ of dimension $\geq 2$ and each $x\in O$  the stabilizer $G_x\subseteq G$ of $x$ acts transitively on $O\!\setminus\!\{x\}.$ 
\end{lemma}

Now we take $G = \SAut(X)$ and $\mathcal{N} = \mathrm{LND}(X).$ Then $\mathcal{N}$ is saturated and by \cite[Corollary 1.21]{5A} there are locally nilpotent derivations $\partial_1, \ldots, \partial_s \in \mathrm{LND}(X)$ which span the tangent space $T_p(Gx)$ at every point $x \in X.$

By \cite[Remark 2.7]{5A} for each $i = 1,\ldots, s$ there is an open dense subset $U_{i}\subseteq X$ which is invariant under the action of the group $H_i = \mathrm{exp}(\BK\partial_i)$ and functions in $\BK[X]^{H_i}$ separate $H_i$-orbits in $U_{i}.$

Let $W = \cap_i U_i$ and $\Omega = \SAut(X)(W).$ Then $\Omega$ is a dense open subset in $X$ and $\partial_i$ has separation property on $\Omega$ for all $i=1,\ldots, s.$ So we obtain the following corollary.

\begin{corollary}\label{cor}
    Let $X$ be an irreducible affine variety. Then there is a dense open $\SAut(X)$-invariant subset $\Omega \subseteq X$ such that for every $\SAut(X)$-orbit $O\subseteq \Omega$ of dimension $\geq 2$ and for each $x \in O$ the stabilizer of $x$ in $\SAut(X)$ acts transitively on $O\!\setminus\!\{x\}.$  
\end{corollary}

\begin{lemma}\label{lemma1}
    Let $X$ be an affine irreducible variety. Suppose that $\mathrm{LND}^*(X) \neq \emptyset.$ Then there is an open dense $\mathrm{SAut}(X)$-invariant subset $\Omega'$ such that for each $x\in \Omega'$ the orbit of $x$ with respect to $\SAut(X)$ coincides with the orbit of $x$ with respect to $\SAut^*(X).$
\end{lemma}

\begin{proof}
Let $Z$ be the set of fixed points in $X$ with respect to $\SAut^*(X).$ Then $Z$ is a closed subset in $X$ and $Z\neq X.$ Note that $Z$ is invariant with respect to all automorphisms of $X$. So $X\!\setminus\!Z$ is an open dense $\SAut(X)$-invariant subset.

Let $\Omega$ be the set from Corollary \ref{cor}. Then $\Omega' = \Omega \cap (X\!\setminus\!Z)$ is a $\SAut(X)$-invariant open dense subset in $X$. Consider a point $x\in \Omega'.$ If $\SAut(X)$-orbit of $x$ is one-dimensional then $\SAut^*(X)$-orbit of $x$ is one-dimensional. They both closed in $X$ so they coincdes.

Suppose that the dimension of $\SAut(X)$-orbit of $x$ is greater then 1. We denote this orbit by $O.$ There is an automorphism $\varphi \in \SAut^*(X)$ such that $y = \varphi(x) \neq x.$ By Corollary \ref{cor} for each $z \neq x \in O$ there is an automorphism $\psi \in \SAut(X)$ such that $\psi(x) = x$ and $\psi(y) = z.$

Then $\psi\varphi\psi^{-1} \in \SAut^*(X)$ and $\psi\varphi\psi^{-1}(x) = z.$ So $\SAut^*(X)$ acts transitevly on $O$ and $\SAut^*(X)$-orbit of $x$ is $O.$
\end{proof}

\begin{theorem}
Let $X$ be an affine irreducible variety and suppose that $\mathrm{LND}^*(X) \neq \emptyset.$ Then $\ML(X) = \ML^*(X).$
\end{theorem}
\begin{proof}

Since $\mathrm{LND}^*(X)$ is non-empty we have $\ML(X) \subseteq \ML^*(X).$ Suppose that $\ML(X) \subsetneq\ML^*(X).$ Let $f \in \ML^*(X)\!\setminus\!\ML(X).$ Then there is $\partial \in \mathrm{LND}(X)$ such that $\partial(f) \neq 0.$ The set $$X_{\partial(f)} = \{x\in X | \partial(f)(x) \neq 0 \}$$
is an open dense subset in $X$ and $f$ is a non-constant function on $\SAut(X)$-orbit of each point in $X_{\partial(f)}$. But $f$ is $\SAut^*(X)$-invariant function. If we take $x \in X_{\partial(f)} \cap \Omega'$ where $\Omega'$ is an open dense subset from Lemma \ref{lemma1} then $f$ is a constant function on $\SAut^*(X)$-orbit of $x$ and non-constant function on $\SAut(X)$-orbit of $x.$ But this orbits coincides and we obtain a contradiction.   
\end{proof}

\section{Modified Derksen invariant}\label{MD}

In this section we provide an example of a variety $X$ with $\HD(X)\neq \HD^*(X)$. This variety is a nonnormal toric variety.

Let $P$ is the monoid obtained from $\mathbb{Z}_{\geq 0}^3$ by removing two vertical rays 
$$
P=\mathbb{Z}_{\geq 0}^3\setminus \left(\{(1,0,n)\mid n\in \mathbb{Z}_{\geq 0}\}\cup\{(0,1,m)\mid m\in \mathbb{Z}_{\geq 0}\}\right).
$$
The corresponding cone is $\sigma^\vee=\QQ_{\geq 0}^3$. So, $\sigma=\QQ_{\geq 0}^3$.
Denote by $X$ the corresponding nonnormal toric variety. It is easy to see that $$\BK[X]=\BK[x^2,xy, y^2, x^3,x^2y,xy^2,y^3,z].$$ 

\begin{theorem}
\rm{i)} $\ML(X)=\BK$;

\rm{ii)} $\ML^*(X)=\BK$;

\rm{iii)} $\HD(X)=\BK[X]$;

\rm{iv)} $\HD^*(X)\neq \BK[X]$.
\end{theorem}
\begin{proof}
Note, that (i) follows from (ii). Also it follows from the fact that $X$ is flexible by \cite[Theorem~1]{BG}. But let us prove it directly. Let us consider three LNDs corresponding to Demazure roots $e_1=(0,0,-1)$, $e_2=(-1,1,0)$ and $e_3=(1,-1,0)$. In coordinates $\partial_{e_1}=\frac{\partial}{\partial z}$, $\partial_{e_2}=y\frac{\partial}{\partial x}$ and $\partial_{e_3}=x\frac{\partial}{\partial y}$. One can see that its kernels are subalgebras 
$$A_i=\bigoplus_{m\in P\cap \pi_i}\BK\chi^m,$$
where $\pi_1$, $\pi_2$ and $\pi_3$ are coordinate planes. We have $\ML(X)\subseteq A_1\cap A_2\cap A_3=\BK$. Hence, $\ML(X)=\BK$.

(ii) It is proved in \cite{B} that for toric varieties $\ML^*(X)=\ML(X)$. This gives the goal. But let us again give a direct proof (according to ideas from \cite{B}). Let us consider three derivations $\partial_1$, $\partial_1+\partial_2$ and $\partial_1+\partial_3$. It is easy to see that all of them are LNDs. Also it is obvious that $\partial_1(z)=\partial_2(z)=\partial_3(z)=1$. Therefore, $\partial_1, \partial_1+\partial_2, \partial_1+\partial_3\in\mathrm{LND}^*(X)$. So, we have 
$$
\ML^*(X)\subseteq \Ker\partial_1\cap\Ker(\partial_1+\partial_2)\cap\Ker(\partial_1+\partial_3)=\Ker\partial_1\cap\Ker\partial_2\cap\Ker\partial_3=\BK.
$$
That is $\ML^*(X)=\BK$.

(iii). To prove (iii) it is sufficient to note that $\BK[X]=(\Ker \partial_1)[z]$ and $z$ is in $\Ker\partial_2$. Therefore, $\BK[X]\subseteq \BK[\Ker\partial_1,\Ker\partial_2]\subseteq \HD(X)$. Therefore, $\HD(X)=\BK[X]$.

(iv) Let $\delta$ be an LND with a slice on $X$. Let us prove that any polynomial $f\in\Ker\delta$ contains the monomial $z$ with zero coefficient. If we prove this, then each polynomial in $\HD^*(X)$ has zero coefficient at $z$. Therefore, $\HD^*(X)\neq\BK[X]=\HD(X)$. 

Let us consider the following $\ZZ$-grading on $\KK[X]$: $\deg x^ay^bz^c=a+b+c$. Let us consider the decomposition of $\delta$ onto the homogeneous components: $\delta=\delta_l+\delta_{l+1}+\ldots+\delta_k$. Then $\delta_l$ is an LND. Let us consider the decomposition of $\delta_l$ onto $M$-homogeneous components. Among these components there is an $M$-homogeneous LND, that corresponds to a Demazure root. But all these $M$-homogeneous components has $\ZZ$-degree $l$. Therefore, $l$ is the $\ZZ$-degree of a Demazure root of $\sigma$. But all Demazure roots of $\sigma$ has one coordinate equals $-1$ and nonnegative other coordinates. Therefore, $\ZZ$-degrees of all Demazure roots are not less then $-1$. Hence, $l\geq -1$.

 Let $s$ be a slice of $\delta$. We can assume that $s$ has zero free term. Indeed, if $\delta(s)=1$, then $\delta(s-\lambda)=1$ for every $\lambda\in\BK$. 
Then the decomposition of $s$ onto the homogeneous components with respect to $\BZ$-grading has the following form: $s=s_t+s_2+\ldots+s_m$, where $t\geq 1$. We have
$$
1=\delta(s)=\left(\sum_{i=l}^k\delta_i\right)\left(\sum_{j=t}^m s_j\right)=\left(\sum_{i=l}^k\sum_{j=t}^m \delta_i(s_j)\right).
$$
The degree of each $\delta_i(s_j)$ is $i+j\geq l+t\geq l+1$. Since $\deg 1=0$, we obtain $l=-1$. 

Note that there are only 3 Demazure roots of $\sigma$ with $\ZZ$-degree equal $-1$. That are $e_1=(0,0,-1)$, $u=(0,-1,0)$ and $v=(0,0,-1)$. But for $u$ and $v$ we have $(P+u)\cap \sigma$ and $(P+v)\cap \sigma$ are not subsets of P. That is there are no $M$-homogeneous LNDs of $X$ with degrees $u$ and $v$. This implies $\delta_l=\mu \partial_{e_1}$ for some $\mu\in \BK\setminus\{0\}$. 

Suppose there exists a polynomial $f\in \Ker\delta$ with a nonzero coefficient at $z$. Let $f = f_0 + \ldots +f_r$ be the decomposition onto homogeneous components of $f$ with respect to $\BZ$-grading. Changing $f$ to $f-\lambda$ we can assume that $f_0=0$. Note that $x,y\notin \BK[X]$, hence, $f_1=cz$, $c\neq 0\in \BK$. Then we have
$$
0=\delta(f)=(\mu \partial_{e_1}+\delta_0+\ldots+\delta_k)(cz+f_2+\ldots+f_r)=c+\sum_{i+j>0}\delta_i(f_j).
$$
We obtain a contradiction. Therefore, all $f\in \Ker\delta$ have zero coefficient at $z$. This finishes the proof of the theorem.
\end{proof}

Since $P$ is not saturated, $X$ is not normal. We do not know any examples of normal varieties admitting LND with a slice with different Derksen and modified Derksen invariants. So, we finish stating the question.
\begin{question}
Is there an example of a normal affine irreducible variety $X$ with conditions $\mathrm{LND}^*(X)\neq\emptyset$ and $\HD^*(X)\neq \HD(X)$?
\end{question}

\end{document}